\documentclass[11pt]{amsart}
\usepackage{amssymb,amsmath,amsthm,latexsym,stmaryrd}

\usepackage[T1]{fontenc}
\usepackage{lmodern}

\newtheorem{theorem}{Theorem}[section]
\newtheorem{definition}[theorem]{Definition}

\newcommand{\ie}{\mbox{i.\hspace{.5pt}e.}\ }

\newcommand{\f}{\varphi}
\newcommand{\g}{\tilde{g}}
\newcommand{\n}{\nabla}
\newcommand{\tn}{\tilde{\n}}
\newcommand{\ttt}{\tilde\tau}
\newcommand{\MM}{\mathcal{M}}

\newcommand{\M}{(\MM,\A\f,\A\xi,\A\eta,\A{}g)}

\newcommand{\I}{\iota}

\newcommand{\R}{\mathbb R}

\newcommand{\F}{\mathcal{F}}
\newcommand{\LL}{\mathcal{L}}

\newcommand{\lm}{\lambda}

\newcommand{\bt}{\beta}
\newcommand{\vt}{\vartheta}

\newcommand{\tlm}{\tilde{\lm}}

\newcommand{\tps}{\tilde{\psi}}

\newcommand{\D}{\mathrm{d}} 

\DeclareMathOperator{\Div}{div} 
\DeclareMathOperator{\tr}{tr} 
\DeclareMathOperator{\Span}{span} 

\newcommand{\A}{\allowbreak{}}

\newcommand{\thmref}[1]{Theorem~\ref{#1}}

\begin{document}
\title[Ricci--Bourguignon Almost Solitons on Sasaki-like Manifolds]
{Ricci--Bourguignon Almost Solitons with Special Potential on Sasaki-like Almost Contact Complex Riemannian Manifolds}

\author[M. Manev]{Mancho Manev $^{1,2}$}

\address{%
$^{1}$ 
Department of Algebra and Geometry;
Faculty of Mathematics and Informatics;
University of Plovdiv Paisii Hilendarski;
24 Tzar Asen St;
4000 Plovdiv, Bulgaria}
\email{mmanev@uni-plovdiv.bg}

\address{$^{2}$ 
Department of Medical Physics and Biophysics;
Faculty of Pharmacy;
Medical University of Plovdiv;
15A Vasil Aprilov Blvd;
4002 Plovdiv, Bulgaria}
\email{mancho.manev@mu-plovdiv.bg}


\begin{abstract}{Almost contact complex Riemannian manifolds, known also as almost contact B-metric manifolds,
are equipped with a pair of pseudo-Riemannian metrics that are mutually associated to each other using the tensor structure.
Here we consider a special class of these manifolds, those of the Sasaki-like type. They have an interesting geometric interpretation: the complex cone of such a manifold is a holomorphic complex Riemannian
manifold (also called a K\"ahler--Norden manifold).
The basic metric on the considered manifold is specialized here as a soliton, \ie has an additional curvature property such that the metric is a self-similar solution of an intrinsic geometric flow.
Almost solitons are more general objects than solitons because they use functions rather than constants as coefficients in the defining condition.
A $\beta$-Ricci-Bour\-gui\-gnon-like almost soliton ($\beta$ is a real constant) is defined using the pair of metrics.
The introduced soliton is a generalization of some well-known (almost) solitons (such as those of Ricci, Schouten, Einstein), which in principle arise from a single metric rather than a pair of metrics.
The soliton potential is chosen to be pointwise collinear to the Reeb vector field,
or the Lie derivative of any B-metric along the potential to be the same metric multiplied by a function.
The resulting manifolds equipped with the introduced almost solitons are characterized geometrically.
Appropriate examples for two types of almost solitons  are constructed and the properties obtained in the theoretical part are confirmed.
}
\end{abstract}



\maketitle


\section{Introduction}

The notion of \emph{Ricci-Bour\-gui\-gnon flow} was introduced by J.\,P. Bour\-gui\-gnon in 1981 \cite{JPB81}.
A time-dependent family of (pseudo-)Riemannian metrics $g(t)$ considered on a smooth manifold $\MM$ is said to evolve through Ricci-Bour\-gui\-gnon flow if $g(t)$ satisfies the following evolution equation
\[
    \frac{\partial}{\partial t} g = -2\left(\rho - \bt \tau g\right),\qquad g(0) = g_0,
\]
where $\bt$ is a real constant, $\rho(t)$ and $\tau(t)$ are the Ricci tensor and the scalar curvature  regarding $g(t)$, respectively.

This flow is an intrinsic geometric flow on $\MM$, whose fixed points or self-similar solutions are its solitons.
The \emph{Ricci-Bour\-gui\-gnon soliton} (in short RB soliton) is described by
the following equation \cite{Cat17,Dwi21}
\begin{equation}\label{RB-g}
  \rho+\frac12 \LL_{\vt} g +(\lm + \bt\tau) g =0,
\end{equation}
where $\LL_{\vt} g$ denotes the Lie derivative of $g$ along the vector field $\vt$ called the soliton potential,
and $\lm$ is the soliton constant.
Briefly, we denote this soliton by $(g, \bt;\vt,\lm)$.
In the case that $\lm$ is a differential function on $\MM$, the solution is called a \emph{RB almost  soliton} \cite{Dwi21}.
A RB soliton is called
expanding if $\lm > 0$, steady if $\lm = 0$ and shrinking if $\lm < 0$.
In case that the soliton potential $\vt$ is a Killing vector field, i.e. $\LL_{\vt} g=0$, the RB soliton is called trivial.

This family of geometric flows contains, the famous Ricci flow for $\bt = 0$,
the Einstein flow for $\bt = \frac12$, the traceless Ricci flow for $\bt = \frac{1}{m}$ and the Schouten flow
for $\bt = \frac{1}{2(m-1)}$, where $m$ is the dimension of the manifold \cite{CatMazMon15,SidSid21}.

For this reason, we consider it more correct to say \emph{$\bt$-RB solitons} and \emph{$\bt$-RB almost solitons}, respectively.

Other recent studies on Ricci-Bour\-gui\-gnon solitons have been done in \cite{BlaTas21,Cun23,Dog23,Mi21,Soy22}.


\section{accR manifolds}


Let us consider a $(2n+1)$-dimensional smooth manifold $\MM$, equipped with an almost contact structure $(\f,\xi,\eta)$
and a B-metric $g$. It is called an \emph{almost contact B-metric man\-i\-fold} or \emph{almost contact complex Riemannian} (abbr.\ \emph{accR}) \emph{manifold} and it is denoted by $\M$.
In detail, $\f$ is an endomorphism
of the tangent bundle $T\MM$, $\xi$ is a Reeb vector field, $\eta$ is its dual contact 1-form and
$g$ is a pseu\-do-Rie\-mannian
metric  of signature $(n+1,n)$ such that
\begin{equation}\label{strM}
\begin{array}{c}
\f\xi = 0,\qquad \f^2 = -\I + \eta \otimes \xi,\qquad
\eta\circ\f=0,\qquad \eta(\xi)=1,\\[4pt]
g(\f x, \f y) = - g(x,y) + \eta(x)\eta(y),
\end{array}
\end{equation}
where $\I$ denotes the identity on $\Gamma(T\MM)$ \cite{GaMiGr}.

In the latter equality and further, $x$, $y$, $z$ 
will stand for arbitrary elements of $\Gamma(T\MM)$ or vectors in the tangent space $T_p\MM$ of $\MM$ at an arbitrary
point $p$ in $\MM$.

The following equations are immediate consequences of \eqref{strM}
\begin{equation}\label{conseq}
g(\f x, y) = g(x,\f y),\qquad g(x, \xi) = \eta(x),\qquad
g(\xi, \xi) = 1,\qquad \eta(\n_x \xi) = 0,
\end{equation}
where $\n$ denotes the Levi-Civita connection of $g$.

The associated metric $\g$ of $g$ on $\MM$ is also a B-metric and it is defined by
\begin{equation}\label{gg}
\g(x,y)=g(x,\f y)+\eta(x)\eta(y).
\end{equation}

The Ganchev--Mihova--Gribachev classification of the investigated manifolds, given in  \cite{GaMiGr},
consists of eleven basic classes $\F_i$, $i\in\{1,2,\dots,11\}$, determined by conditions for
the (0,3)-tensor $F$ defined by
\begin{equation}\label{F=nfi}
F(x,y,z)=g\bigl( \left( \nabla_x \f \right)y,z\bigr).
\end{equation}
It has the following basic properties:
\begin{eqnarray}\label{F-prop1}
&F(x,y,z)=F(x,z,y)
=F(x,\f y,\f z)+\eta(y)F(x,\xi,z)
+\eta(z)F(x,y,\xi),\\[4pt] \label{F-prop2}
&F(x,\f y, \xi)=(\n_x\eta)y=g(\n_x\xi,y).
\end{eqnarray}

%


\subsection{Sasaki-like accR manifolds}

An interesting class of accR manifolds was introduced in \cite{IvMaMa45} by the condition that the complex cone of such a manifold is a K\"ahler-Norden manifold. They are called \emph{Sasaki-like manifolds} and are defined by the condition
\begin{equation}\label{Sl}
F(x,y,z)=g(\f x,\f y)\eta(z)+g(\f x,\f z)\eta(y).
\end{equation}
The class of Sasaki-like manifolds is contained in the basic class $\F_4$ of the Ganchev--Mihova--Gribachev classification, not intersecting with the special class $\F_0$ of cosymplectic accR manifolds defined by $F=0$.

Moreover,
the following identities are valid for this type of accR manifolds \cite{IvMaMa45}
\begin{equation}\label{curSl}
\n_x \xi=-\f x, \qquad 
\rho(x,\xi)=2n\, \eta(x),\qquad 		
\rho(\xi,\xi)=2n,
\end{equation}
where 
$\rho$ stands for 
the Ricci tensor for $g$.

Let $\tau$ and $\ttt$ be the scalar curvatures with respect to $g$ and $\g$, respectively,
and let $\tau^*$ be the associated quantity of $\tau$ regarding $\f$, defined by
$\tau^* = g^{ij} \rho(e_i ,\f e_j )$.
Then, for a Sasaki-like manifold we have
\begin{equation}\label{tttSl}
 \ttt = -\tau^* + 2n.
\end{equation}

\subsection{Einstein-like accR manifolds}

In \cite{Man62}, it is introduce the following notion.
An accR manifold $\M$ is said to be
\emph{Einstein-like} if its Ricci tensor $\rho$ satisfies
\begin{equation}\label{defEl}
\begin{array}{l}
\rho=a\,g +b\,\g +c\,\eta\otimes\eta
\end{array}
\end{equation}
for some triplet of constants $(a,b,c)$.
%
In particular, when $b=0$ and $b=c=0$, the manifold is called an \emph{$\eta$-Einstein manifold} and an \emph{Einstein manifold}, respectively.
If $a$, $b$, $c$ in \eqref{defEl} are functions on $\MM$, then the manifold is called \emph{almost Einstein-like}, \emph{almost $\eta$-Einstein} and \emph{almost Einstein}, respectively \cite{Man64}.

Consequences of \eqref{defEl} are the following
$
\tau=(2n+1)a +b +c
$ 
 and $\tau^*=-2nb$.


\section{$\bt$-RB almost solitons}

A generalization of the known RB soliton on a manifold with an additional 1-form $\eta$ is an \emph{$\eta$-Ricci-Bour\-gui\-gnon
soliton} defined following \eqref{RB-g} by
\begin{equation}\label{etaRB}
  \rho+\frac12 \LL_{\vt} g +(\lm + \bt\tau) g + \mu \eta\otimes\eta=0,
\end{equation}
where $\mu$ is also a constant \cite{BlaTas21}. Obviously, an $\eta$-Ricci-Bour\-gui\-gnon soliton with $\mu = 0$ is a RB soliton.
Again, in the case where $\lm$ and $\mu$ are functions on the manifold, almost solitons of the corresponding kind are said to be given.

In the present paper, we study an accR manifold.
Having two B-metrics $g$ and $\g$ related to each other with respect to the structure of such a manifold gives us reason to introduce a more natural generalization of the $\bt$-RB soliton than \eqref{etaRB}.
In addition, we also have the structure 1-form $\eta$ so that $\eta\otimes\eta$ is included in both B-metrics $g$ and $\g$ as their restriction on the vertical distribution $\Span(\xi)$.

\begin{definition}
An accR manifold $\M$ is called a
\emph{$\bt$-Ricci-Bour\-gui\-gnon-like soliton} (in short \emph{$\bt$-RB-like soliton}) with potential vector field $\vt$ if its Ricci tensor $\rho$ satisfies the following condition for a pair of constants $(\lm,\tlm)$
\begin{equation}\label{defRBl}
\rho + \frac12 \mathcal{L}_{\vt} g + \frac12 \mathcal{L}_{\vt} \g  + (\lm+\bt\tau) g  + (\tlm+\bt\ttt) \g  =0,
\end{equation}
where $\ttt$ is the scalar curvature of the manifold with respect to $\g$ and the corresponding Levi-Civita connection $\tn$.
If $(\lm,\tlm)$ is a pair of functions on $\MM$ satisfying \eqref{defRBl}, then $\M$ is called \emph{$\bt$-Ricci-Bour\-gui\-gnon-like almost soliton} (in short \emph{$\bt$-RB-like almost soliton}).
\end{definition}

By taking the trace in \eqref{defRBl} with respect to $g$, we get
\[
[1+ (2n+1)\bt]\tau+\bt\ttt +\Div_g\vt +\frac12\tr_g\left(\mathcal{L}_{\vt}\g\right) + (2n+1)\lm + \tlm=0
\]
by means of the formula $\Div_g\vt=\frac12 \tr_g\left(\mathcal{L}_{\vt}g\right)= g^{ij}g\left(\n_{e_i}\vt,e_j \right)$.
The expression of the trace in the above equality is the following
\[
\frac12 \tr_g\left(\mathcal{L}_{\vt}\g\right)
= g^{ij}\g\left(\tn_{e_i}\vt,e_j \right).
\]

\subsection{The potential is conformal vector field}

Recall that a vector field,  e.g. the potential $\vt$, on $\MM$ is called a \emph{conformal vector field with respect to $g$} if there exists
a function $\psi$ on $\MM$ such that \cite{Dwi21}
\[
\LL_{\vt} g = 2\psi g.
\]
The conformal vector field is
nontrivial if $\psi \neq 0$. If $\psi = 0$, then $\vt$ is called a \emph{Killing vector field} with respect to $g$.

Similarly, $\vt$ is called a \emph{conformal vector field with respect to $\g$} if there exists
a function $\tilde\psi$ on $\MM$ such that
\[
\LL_{\vt} \g = 2\tilde\psi \g.
\]
Depending on whether $\tps$ is nonzero or zero, we have a vector field $\vt$ that is nontrivial scalar or Killing, respectively.

\begin{theorem}\label{thm:conf}
Let $\M$ be a $(2n+1)$-dimensional Sasaki-like accR manifold that is a $\bt$-RB-like almost soliton with a pair of soliton functions $(\lm,\tlm)$ and conformal potential vector field $\vt$ with potential functions $\psi$ and $\tps$ respect to $g$ and $\g$, respectively.
Then the manifold  is Einstein-like and has the following Ricci tensor:
\begin{equation}\label{concl-go-Sl}
\rho=\left(\frac{\tau}{2n}-1\right)g+\left(\frac{\ttt}{2n}-1\right)\g
\end{equation}
and the following property is valid
\begin{equation}\label{concl-tttt-Sl}
\tau+\ttt=4n(n+1).
\end{equation}

If $\bt\neq-\frac{1}{2n}$, then the scalar curvatures with respect to $g$ and $\g$ can be expressed separately as
\begin{equation}\label{concl-tau-Sl}
\tau=-\frac{2n}{1+2n\bt}(\psi+\lm-1),\qquad
\ttt=-\frac{2n}{1 + 2n\bt}(\tps + \tlm-1)
\end{equation}
and the following condition for the used functions is valid
\begin{equation}\label{rel-Sl}
\psi+\lm + \tps+\tlm +2n[1+2(n+1)\bt]=0.
\end{equation}

If $\bt=-\frac{1}{2n}$, then the following properties hold
\begin{equation}\label{func-case}
\psi + \lm=1,\qquad  \tps + \tlm=1.
\end{equation}
%
\end{theorem}
\begin{proof}
Under these circumstances regarding the considered manifold,
due to \eqref{defRBl} its Ricci tensor for $g$ has the following form
\begin{equation}\label{ro-psi}
 \rho=-(\psi + \lm + \bt\tau) g -(\tps + \tlm + \bt\ttt) \g.
\end{equation}
Taking the trace of the last expression, we obtain that the scalar curvatures of $\M$ with respect to $g$ and $\g$ are related as follows
\begin{equation}\label{tttt}
[1+(2n+1)\bt]\tau + \bt\ttt + (2n+1)(\psi+\lm) +\tps + \tlm=0.
\end{equation}

Using \eqref{ro-psi},
we take the appropriate trace to obtain $\tau^*$ on the left-hand side and the resulting relation to $\ttt$  is
\begin{equation}\label{tttt*}
\tau^*=2n(\tps + \tlm + \bt\ttt).
\end{equation}
By virtue of \eqref{tttSl} for a Sasaki-like accR manifold and \eqref{tttt*}, we obtain the expression of the scalar curvature with respect to $\g$ as follows
\begin{equation}\label{tttt2-Sl}
\ttt=-2n\frac{\tps + \tlm-1}{1 + 2n\bt}.
\end{equation}
The last formula is true for the case $\bt\neq-\frac{1}{2n}$.


Combining \eqref{tttt} and \eqref{tttt2-Sl}, we get the following form of the scalar curvature with respect to $g$
\begin{equation}\label{tau-Sl}
\tau=-\frac{1}{1+(2n+1)\bt}\left\{
(2n+1)(\psi+\lm)
+1+\frac{\tps+\tlm-1}{1+2n\bt}\right\},
\end{equation}
where $\bt\neq -\frac{1}{2n+1}$. Otherwise, for $\bt= -\frac{1}{2n+1}$ the identity in \eqref{rel-Sl} is valid for this value of $\bt$.


On the other hand, a consequence of \eqref{ro-psi} and the value of $\rho(\xi,\xi)$ from \eqref{curSl} for the Sasaki-like case gives the following relation
\begin{equation}\label{tttt-Sl}
\bt\tau+\bt\ttt+\psi + \lm+\tps + \tlm+2n=0.
\end{equation}

Let us check what follows in the particular case $\bt=0$.
Then, the last relation implies
\[ 
\psi + \lm+\tps + \tlm+2n=0,
\] 
which we apply in \eqref{tau-Sl} and \eqref{tttt2-Sl} to specialize them in the following form
\begin{equation}\label{tttt-bt=0-Sl}
\tau=-2n(\psi+\lm-1),\qquad \ttt=-2n(\tps + \tlm-1).
\end{equation}
    Therefore, the identity in \eqref{concl-tttt-Sl} holds, as do the formulas in \eqref{concl-tau-Sl}.
    Thus, we find that no different results are obtained for $\bt=0$ compared to the case $\bt\neq -\frac{1}{2n}$.

Let us return to \eqref{tttt-Sl} and the solution of the system of equations in \eqref{tttt-Sl} and \eqref{tttt} with respect to $\tau$ and $\ttt$ for $\bt\neq 0$ and $\bt\neq -\frac{1}{2n}$ gives us
\begin{equation}\label{bt-tt-Sl}
\tau=-2n\frac{\psi+\lm-1}{1+2n\bt},\qquad
\ttt=-\frac{1}{\bt}\left\{\frac{\psi+\lm-1}{1+2n\bt}+\tps+\tlm+2n+1\right\}.
\end{equation}
After that comparing the equalities for $\tau$ in \eqref{tau-Sl} and \eqref{bt-tt-Sl}, we obtain \eqref{rel-Sl}.
The same relation results from a similar comparison for $\ttt$ in \eqref{tttt2-Sl} and \eqref{bt-tt-Sl}.
The identity in \eqref{rel-Sl} is a generalization of the corresponding result in the $\bt=0$ case. In this way, we obtain the formulas in \eqref{concl-tau-Sl}.

Due to \eqref{rel-Sl} the sum of $\tau$ and $\ttt$ is a constant that depends only on the dimension of the manifold as in \eqref{concl-tttt-Sl}.

For the case $\bt=-\frac{1}{2n}$, the equalities in \eqref{concl-tttt-Sl} and \eqref{rel-Sl} 
imply the relations in \eqref{func-case}. Therefore, \eqref{tttt} and \eqref{tttt-Sl} take the form in \eqref{concl-tttt-Sl}.

In conclusion, using \eqref{ro-psi} and \eqref{concl-tau-Sl}, we obtain the Ricci tensor expression in \eqref{concl-go-Sl}, which shows that the manifold is Einstein-like since \eqref{defEl} is satisfied for $a=\frac{\tau}{2n}-1$, $b=\frac{\ttt}{2n}-1$ and $c=0$.
\end{proof}

\subsubsection{Example of a $\bt$-RB almost soliton with a conformal potential}

As in \cite{IvMaMa45}, let us consider
a Sasaki-like accR manifold of arbitrary dimension with an Einstein metric $\bar g$. The image of this manifold by a contact homothetic
transformation of the metric given by
$g = p \bar g + q \tilde{\bar{g}} + ( 1 - p - q )\eta\otimes\eta$ for $p,q\in\R$, $(p,q)\neq(0,0)$,
is also a Sasaki-like accR manifold. The corresponding Ricci tensor has the form
\[
\rho=\frac{2n}{p^2+q^2}\left\{p g - q \g + \left(p^2+q^2-p\right)\eta\otimes\eta\right\}.
\]

Now we calculate the scalar curvatures with respect to B-metrics $g$ and $\g$ as follows
\begin{equation}\label{Ex1-ttt}
\tau= 2n\left\{1+\frac{2np}{p^2+q^2}\right\},\qquad
\ttt=2n\left\{1-\frac{2nq}{p^2+q^2}\right\}.
\end{equation}

Let $\vt$ be a conformal vector field with respect to both B-metrics $g$ and $\g$ with functions $\psi$ and $\tilde\psi$, respectively.
We construct a $\bt$-RB-like almost soliton on the transformed manifold with potential $\vt$ and a pair of functions $(\lm,\tlm)$ satisfying \eqref{defRBl}.

By virtue of  \eqref{concl-tttt-Sl} and \eqref{Ex1-ttt}, we obtain the following condition $p^2+q^2-p+q=0$, which has a solution $p=\frac12(1+\sqrt{2}\cos{t})$, $q=-\frac12(1-\sqrt{2}\sin{t})$ for $t\in\R$. Then the scalar curvatures from \eqref{Ex1-ttt} specialize into the following form for $t\neq (8l+3)\frac{\pi}{4}$, $l\in\mathbb{Z}$
\begin{equation}\label{Ex1-ttt=}
\begin{array}{l}
\tau= 2n\dfrac{(n+1)\sqrt{2} +(2n+1) \cos{t}-\sin{t}}{\sqrt{2} +\cos{t}-\sin{t}},\\
\ttt=2n\dfrac{(n+1)\sqrt{2}+\cos{t} -(2n+1) \sin{t}}{\sqrt{2} +\cos{t}-\sin{t}}.
\end{array}
\end{equation}

Then we determine the functions $\psi$, $\tilde\psi$, $\lm$, $\tilde\lm$ in the case $\bt\neq-\frac{1}{2n}$ as follows
\begin{equation}\label{Ex1-psilm}
\begin{array}{l}
\psi+\lm = 1-(1+2n\bt)\dfrac{(n+1)\sqrt{2} +(2n+1) \cos{t}-\sin{t}}{\sqrt{2} +\cos{t}-\sin{t}},\\
\tilde\psi+\tilde\lm =1-(1+2n\bt)\dfrac{(n+1)\sqrt{2}+\cos{t} -(2n+1) \sin{t}}{\sqrt{2} +\cos{t}-\sin{t}}.
\end{array}
\end{equation}
In the case $\bt=-\frac{1}{2n}$, the equalities in \eqref{Ex1-psilm} obviously reduce to those in  \eqref{func-case}.

We directly verify that \eqref{Ex1-ttt=} and \eqref{Ex1-psilm} satisfy  the expressions in \eqref{concl-tau-Sl} and \eqref{rel-Sl}.

In conclusion, we found that the constructed manifold satisfies the conditions of \thmref{thm:conf}.


\subsection{The potential is vertical vector field}

Suppose that $\M$ is a Sasaki-like accR manifold admitting a $\bt$-RB-like almost soliton whose potential vector field $\vt$ is pointwise collinear with $\xi$, i.e.
$\vt = k\xi$, where $k$ is a differentiable function on $\MM$. It is clear that $k = \eta(\vt)$
and therefore $\vt$ belongs to the vertical distribution $H^{\bot} = \Span(\xi)$, which is
orthogonal to the contact distribution $H = \ker(\eta)$ with respect to both $g$ and $\g$.

According to \cite{Man63},  we have the expression  $\left(\LL_{\vt} g\right)(x,y)
=h(x,y)-2kg(x,\f y)$ for the vertical potential $\vt$, where the first equality of \eqref{curSl} is used and the symmetric tensor $h(x,y)=\D{k}(x)\eta(y)+\D{k}(y)\eta(x)$ is denoted for the sake of brevity. Then, by virtue of \eqref{gg} we get
\begin{equation}\label{LLg}
\LL_{\vt} g
=h-2k\left(\g-\eta\otimes\eta\right).
\end{equation}

Similarly, since for a Sasaki-like accR manifold $\tn_x\xi=-\f x$ is true \cite{Man78}, we have $\left(\LL_{\vt} \g\right)(x,y)
=h(x,y)-2k g(\f x,\f y)$ and considering the last equality of \eqref{strM}, the last expression  takes the following form
\begin{equation}\label{LLg2}
\LL_{\vt} \g
=h+2k \left(g-\eta\otimes\eta\right).
\end{equation}


\begin{theorem}\label{thm:vert}
Let $\M$ be a $(2n+1)$-dimensional Sasaki-like accR manifold that is a $\bt$-RB-like almost soliton with a pair of soliton functions $(\lm,\tlm)$ and vertical potential vector field $\vt$ with potential function $k$.
Then the manifold is Einstein-like and has the following Ricci tensor
\begin{equation}\label{rho-tttt-vert-Sl}
\rho=\left(\frac{\tau}{2n}-1\right)g
+\left(\frac{\ttt}{2n}-1\right)\g
-\left\{\frac{\tau+\ttt}{2n}-2(n+1)\right\}\eta\otimes\eta.
\end{equation}

In the case of $\bt\neq-\frac{1}{2n}$,
the scalar curvatures with respect to $g$ and $\g$ are determined by:
\begin{equation}\label{concl-vert-tau-Sl}
\tau=-\frac{2n}{1+2n\bt}(\lm+k-1),\qquad
\ttt=-\frac{2n}{1 + 2n\bt}(\tlm -k -1).
\end{equation}
Moreover, the following condition for the used functions is valid
\begin{equation}\label{func-vert-Sl}
\D{k}=\D{k}(\xi)\eta,\qquad
\D{k}(\xi) = -\frac{\lm+\tlm-2}{2(1+2n\bt)}-n-1.
\end{equation}

In the case of $\bt=-\frac{1}{2n}$,
the two scalar curvatures satisfy the following relation
\begin{equation}\label{tttt2-vert-Sl}
\tau+\ttt=4n\left\{\D{k}(\xi)+n+1\right\}
\end{equation}
and the soliton functions are expressed by $k$ as follows
\begin{equation}\label{lmtlmk-vert-Sl}
\lm=1-k,\qquad \tlm=1+k.
\end{equation}
\end{theorem}
\begin{proof}

Replacing \eqref{LLg} and  \eqref{LLg2} in \eqref{defRBl}, we obtain
\begin{equation}\label{rho-vert-Sl}
\rho = - (\lm+\bt\tau+k) g  - (\tlm+\bt\ttt-k) \g -h.
\end{equation}

For a Sasaki-like accR manifold we know the expression of $\rho(x,\xi)$ from \eqref{curSl}. We then compare it with the corresponding consequence of \eqref{rho-vert-Sl} and obtain $\D{k}=\D{k}(\xi)\eta$, where
\begin{equation}\label{rho-vert-Sl-2}
\D{k}(\xi) = -\frac12 \left\{\lm+\tlm+\bt(\tau+\ttt) +2n\right\}.
\end{equation}
Therefore, $k$ is a horizontal constant and $h$ takes the following form
\begin{equation}\label{h-vert-Sl}
h=-\left\{\lm+\tlm+\bt(\tau+\ttt) +2n\right\}\eta\otimes\eta.
\end{equation}
Then,
applying the last equality and \eqref{rho-vert-Sl-2} in \eqref{rho-vert-Sl},
for the Ricci tensor $\rho$ we get
\begin{equation}\label{rho-vert-Sl-3}
\rho = - (\lm+\bt\tau+k) g  - (\tlm+\bt\ttt-k) \g +\left\{\lm+\tlm+\bt(\tau+\ttt) +2n\right\}\eta\otimes\eta.
\end{equation}

Taking the appropriate traces in \eqref{rho-vert-Sl-3}, we obtain for $\bt\neq -\frac{1}{2n}$ the first equality in \eqref{concl-vert-tau-Sl} as well as
\begin{equation}\label{tau-ttt-vert-Sl}
\tau^*=2n(\tlm+\bt\ttt-k).
\end{equation}
Then, bearing in mind \eqref{tttSl}, the equality in \eqref{tau-ttt-vert-Sl} implies for $\bt\neq -\frac{1}{2n}$
the second equality in \eqref{concl-vert-tau-Sl}.
Substituting the expressions of $\tau$ and $\ttt$ from \eqref{concl-vert-tau-Sl}   into \eqref{rho-vert-Sl-3}, we get \eqref{rho-tttt-vert-Sl}.

As a consequence of \eqref{concl-vert-tau-Sl} and \eqref{rho-vert-Sl-2}, we express $\tau+\ttt$ in two ways. One is given in \eqref{tttt2-vert-Sl} and the other is as follows
\begin{equation}\label{tttt-dk-vert-Sl=}
\tau+\ttt=-\frac{2n}{1+2n\bt}\left(\lm+\tlm-2\right).
\end{equation}
Comparing two expressions implies the relation in \eqref{func-vert-Sl} between the used functions and the constant $\bt$.

Comparing \eqref{rho-vert-Sl-3} with \eqref{defEl}, it follows that $\M$ is Einstein-like and the coefficients in \eqref{defEl} are the following
\begin{equation}\label{abc-vert-Sl}
a=\frac{\tau}{2n}-1,\qquad b=\frac{\ttt}{2n}-1,\qquad c=-\frac{\tau+\ttt}{2n}+2(n+1).
\end{equation}

In the particular case $\bt= -\frac{1}{2n}$,  the dependencies in \eqref{lmtlmk-vert-Sl}  hold due to \eqref{concl-vert-tau-Sl}. Then, \eqref{rho-vert-Sl-3} is specialized as in \eqref{rho-tttt-vert-Sl}.
In addition, according to \eqref{rho-vert-Sl-2}, the equality in \eqref{tttt2-vert-Sl} also holds.
Note that in this case the scalar curvatures with respect to each of the B-metrics cannot be expressed separately.
\end{proof}

\subsubsection{Example of a $\bt$-RB almost soliton with a vertical potential}

Let us consider an explicit example given as Example 2 in \cite{IvMaMa45}. It concerns a Sasaki-like accR manifold derived on a Lie group $G$ of
dimension $5$, i.e. for $n=2$,
 with a basis of left-invariant vector fields $\{e_0,\dots, e_{4}\}$. The corresponding Lie algebra is
 defined by the commutators
\begin{equation}\label{comEx1}
\begin{array}{lll}
[e_0,e_1] = p e_2 + e_3 + q e_4,\qquad &[e_0,e_2] = - p e_1 -
q e_3 + e_4,\qquad &\\[0pt]
[e_0,e_3] = - e_1  - q e_2 + p e_4,\qquad &[e_0,e_4] = q e_1
- e_2 - p e_3,\qquad & p,q\in\R.
\end{array}
\end{equation}
The introduced accR structure is defined as follows
\begin{equation}\label{strEx1}
\begin{array}{l}
g(e_0,e_0)=g(e_1,e_1)=g(e_2,e_2)=-g(e_{3},e_{3})=-g(e_{4},e_{4})=1,
\\[0pt]
g(e_i,e_j)=0,\quad
i,j\in\{0,1,\dots,4\},\; i\neq j,
\\[0pt]
\xi=e_0, \quad \f  e_1=e_{3},\quad  \f e_2=e_{4},\quad \f  e_3=-e_{1},\quad \f  e_4=-e_{2}.
\end{array}
\end{equation}

Then, in \cite{Man62}, the components of the curvature tensor $R_{ijkl}=R(e_i,e_j,e_k,e_l)$
and those of the Ricci tensor  $\rho_{ij}=\rho(e_i,e_j)$ are calculated.
The non-zero of them are determined by the following equalities  and the property $R_{ijkl}=-R_{jikl}=-R_{ijlk}$:
\begin{equation}\label{Rex1}
\begin{array}{l}
R_{0110}=R_{0220}=-R_{0330}=-R_{0440}=1,\\
R_{1234}=R_{1432}=R_{2341}=R_{3412}=1,\\
R_{1331}=R_{2442}=1,\qquad
\rho_{00}=4.
\end{array}
\end{equation}
Therefore, its Ricci tensor has the form $\rho=4\eta\otimes \eta$ and the manifold is $\eta$-Einstein. Hence, the scalar curvature of $g$ is $\tau=4$ and the constructed manifold is  $*$-scalar flat, i.e. $\tau^*=0$ \cite{Man73}.
Then, due to  \eqref{tttSl}, we obtain for the scalar curvature of $\g$ the value $\ttt=4$.

Using \eqref{defRBl}, let us construct a $\bt$-RB-like almost soliton on $(G,\f,\xi,\eta,g)$ with vertical potential $\vt = k\xi$ for a constant $\bt$ and a pair of functions $(\lm,\tlm)$.

From \eqref{concl-vert-tau-Sl}, we determine the following conditions in the case $\bt\neq-\frac{1}{2n}$
\begin{equation}\label{Ex2-lmkbt}
\lm +k=-4\bt,\qquad \tlm -k=-4\bt.
\end{equation}
Then, \eqref{func-vert-Sl} implies $\D{k}(\xi)=-2$ and $\D{k}=-2\eta$. A solution of the last equation is e.g. $k=-2t$ assuming $\eta=\D{t}$. This form of $k$ also satisfies the condition in \eqref{tttt2-vert-Sl} for the case $\bt=-\frac{1}{2n}$.
This allows us to determine functions $(\lm,\tlm)$  from \eqref{Ex2-lmkbt} and \eqref{lmtlmk-vert-Sl} for all values of $\bt$ by
\begin{equation}\label{Ex2-lmbt}
\lm =2\left(t-2\bt\right),\qquad \tlm =-2\left(t+2\bt\right).
\end{equation}
Hence and from \eqref{h-vert-Sl} we obtain $h=-4\eta\otimes\eta$. As a consequence, \eqref{LLg} and \eqref{LLg2} take the following form
\[
\LL_{\vt} g
=4t\g-4(t+1)\eta\otimes\eta,\qquad
\LL_{\vt} \g
=-4t g+4(t+1)\eta\otimes\eta.
\]
Finally, we found that all the findings in \thmref{thm:vert} are satisfied for the given example.


\vspace{6pt}




\section*{Acknowledgments}{The research is partially supported by project FP23-FMI-002
of the Scientific Research Fund, University of Plovdiv Paisii Hilendarski.}

\end{document}